\documentclass[reqno, 10pt]{amsart}
\usepackage{amssymb}
\usepackage{amsmath}
\usepackage[mathscr]{euscript}
\usepackage{hyperref}
\usepackage{graphicx}
\usepackage[normalem]{ulem}

\usepackage{amssymb}
\usepackage{hyperref}
\RequirePackage[dvipsnames]{xcolor} 
\definecolor{halfgray}{gray}{0.55} 
\definecolor{webgreen}{rgb}{0,0.5,0}
\definecolor{webbrown}{rgb}{.6,0,0} \hypersetup{%
  colorlinks=true, linktocpage=true, pdfstartpage=3,
  pdfstartview=FitV,%
  breaklinks=true, pdfpagemode=UseNone, pageanchor=true,
  pdfpagemode=UseOutlines,%
  plainpages=false, bookmarksnumbered, bookmarksopen=true,
  bookmarksopenlevel=1,%
  hypertexnames=true,
  pdfhighlight=/O,
  urlcolor=webbrown, linkcolor=RoyalBlue,
  citecolor=webgreen, 
  pdftitle={},%
  pdfauthor={},%
  pdfsubject={2000 MAthematical Subject Classification: Primary:},%
  pdfkeywords={},%
  pdfcreator={pdfLaTeX},%
  pdfproducer={LaTeX with hyperref}%
}

\newtheorem{theorem}{Theorem}[section]
\newtheorem{lemma}[theorem]{Lemma}
\newtheorem{corollary}[theorem]{Corollary}
\newtheorem{proposition}[theorem]{Proposition}

\theoremstyle{definition} 
\newtheorem{definition}[theorem]{Definition}

\newtheorem{remark}[theorem]{Remark}

\def\R{\mathbb{R}}
\def\N{\mathbb{N}}
\def\Z{\mathbb{Z}}

\newcommand{\norm}[1]{{\left\lVert \, #1 \, \right\rVert}}

\def\Id{\text{Id}}
\def\Aut{\text{Aut}}

\begin{document}

\title
[Cohomology of twisted cocycles]
{Cohomology of fiber-bunched twisted cocycles over hyperbolic systems}

\author{Lucas Backes}

\address{Departamento de Matem\'atica, Universidade Federal do Rio Grande do Sul, Av. Bento Gon\c{c}alves 9500, CEP 91509-900, Porto Alegre, RS, Brazil.}

\email{lucas.backes@ufrgs.br }


\keywords{Twisted cocycles, Cohomology, Hyperbolic Systems, Periodic points}
\subjclass[2010]{Primary: 37H05, 37A20; Secondary: 37D20}

\begin{abstract}
A twisted cocycle taking values on a Lie Group $G$ is a cocycle that, in each step, is twisted by an automorphism of $G$. In the case when $G=GL(d,\mathbb{R})$, we prove that if two H\"older continuous twisted cocycles satisfying the so called fiber-bunching condition have the same periodic data then they are cohomologous.   
\end{abstract}

\maketitle

\section{Introduction}

Given a homeomorphism $f:M\to M$ acting on compact metric space $(M,d)$ and an automorhism $\alpha\in \Aut(G)$ of a topological group $G$, we say that the map $A_\alpha:\Z\times M\to G$ is an \emph{$\alpha$-twisted cocycle} over $f$ if
\begin{equation}\label{eq: twisted cocycle eq intro}
A_\alpha^{m+n}(x)=A_\alpha^n(f^m(x)) \alpha^n( A_\alpha^m(x))
\end{equation}
for all $x\in M$ and $m,n \in \Z$. 

Two $\alpha$-twisted cocycles $A_\alpha$ and $B_\alpha$ over $f$ are said to be \emph{$\alpha$-cohomologous} whenever there exists a \emph{transfer map} map $P:M\to G$ satisfying 
\begin{equation*}
A^n_\alpha (x)=P(f^n(x))B^n_\alpha(x)\alpha^n(P(x))^{-1}
\end{equation*}
for every $x\in M$ and $n\in \Z$. Observe that in the case when $\alpha=\Id$ the notions of $\alpha$-twisted cocycle and $\alpha$-cohomology coincide with the ``standard" notions of cocycles and cohomology in Dynamical Systems \cite{KN11}.

Cohomology of twisted cocycles appears naturally in many problems in Dynamics. For instance, any map $A:M\to G$ naturally generates an $\alpha$-twisted cocycle $A_\alpha$ over $f$ (see Section \ref{sec: twisted}). In this case, we can consider the \emph{twisted skew-product} $F_{A,\alpha}:M\times G\to M\times G$ given by $F_{A,\alpha}(x,g)=(f(x),A(x)\alpha(g))$. Now, the problem of determining whether two twisted skew-products $F_{A,\alpha}$ and $F_{B,\alpha}$ are conjugated reduces to the problem of studying whether $A_\alpha$ and $B_\alpha$ are $\alpha$-cohomologous. In fact, the map $U(x,g)=(x,P(x)g)$ conjugates $F_{A,\alpha}$ and $F_{B,\alpha}$ precisely when $P$ is a transfer map for $A_\alpha$ and $B_\alpha$. This observation applied to the case when $G=GL(d,\R)$ is what motivates much of this note. Other applications also appear in the study of regularity of the transfer map for non-abelian cocycles over Anosov actions \cite{NT98}, in applications to the differentiable rigidity of Anosov diffeomorphisms \cite{dlL87} and the study of local rigidity of higher rank abelian partially hyperbolic actions \cite{DK10}. For more applications we refer to Section 4.6 of \cite{KN11} and to \cite{Kon95}.

In the present paper we are interested in describing necessary and sufficient conditions under which two $\alpha$-twisted cocycles $A_\alpha$ and $B_\alpha$ are $\alpha$-cohomologous whenever $f$ is a hyperbolic map. In the case when $\alpha=\Id$ and $G$ is an abelian group admiting a bi-invariant metric, a first criterion was given by Liv\v{s}ic in his seminal papers \cite{Liv71} and \cite{Liv72}. More precisely, he proved that $A_\Id$ and $B_\Id$ are $\Id$-cohomologous if and only if 
\begin{displaymath}
A^n_\Id(p)=B^n_\Id(p) \text{ for every } p\in \text{Fix}(f^n).
\end{displaymath}

Because of its many applications, still in the case when $\alpha=\Id$, this criterion was extended by many authors to many different settings usually eliminating the assumptions that $G$ is abelian and admits a bi-invariant metric. See for instance \cite{AKL18,Bac15,BK16,Kal11,Par99,Sad15,Sch99}.

The case when $\alpha$ is not the identity, on the other hand, despite of its many applications, has received much less attention. To the best of the author's knowledge, the best result in this setting is a theorem by Walkden \cite{Wal00} where he got an analogous result to the original Liv\v{s}ic's theorem under the assumptions that $G$ is a connected Lie group admitting a bi-invariant metric\footnote{In the case when $G=GL(d,\mathbb{R})$ the existence of the bi-invariant metric can be replaced by a \emph{bounded distortion condition}. See comments after Theorem \ref{teo: main}} and the automorphism $\alpha$ satisfies some ``growth" conditions. The objective of this paper is to extend the results of \cite{Wal00} to the case when $G=GL(d,\mathbb{R})$.

\subsection{Main results} \label{sec: main results}

The main result of this work is the following one (see Section \ref{sec: prelimin} for precise definitions):

\begin{theorem} \label{teo: main} 
Let $f: M\to M$ be a Lipschitz continuous transitive hyperbolic homeomorphism on a compact metric space
$(M,d)$, $A,B:M\to GL(d,\mathbb{R})$ two $\nu$-H\"older continuous maps and $\alpha \in \Aut(GL(d,\mathbb{R}))$ be an automorphism of $GL(d,\mathbb{R})$. Suppose that the twisted cocycles $A_\alpha$ and $B_\alpha$ are fiber-bunched. Moreover, suppose that they satisfy the periodic orbit condition  \begin{equation} \label{eq: periodic orbit cond main teo}
A^n_\alpha(p)=B^n_\alpha(p), \quad\forall n\in \Z,\ \forall p\in \text{Fix}(f^n). 
\end{equation}
Then, there exists a $\nu$-H\"older continuous map $P:M\to GL(d,\mathbb{R})$ such that
\begin{equation}\label{eq: cohomo eq main teo}
A^n_\alpha(x)=P\big(f^n(x)\big) B_\alpha^n (x)\alpha^n\left(P(x)\right)^{-1}, \quad\forall x\in M,\ \forall n\in\Z.
\end{equation}
\end{theorem}
This result consists of a generalization of the main results of \cite{Bac15} and \cite{Sad15} to the case of twisted cocycles. In fact, the main result of those works can be obtained as corollaries of the previous one by taking $\alpha=\Id$. Moreover, this result also generalizes the main result of \cite{Wal00} in the case when $G=GL(d,\mathbb{R})$. Indeed, it was observed in \cite[Remark 3.4]{Wal00} that in such case, instead asking for the group to admit a bi-invariant metric (recall that $GL(d,\mathbb{R})$ does not admit such a metric), one can assume some \emph{bounded distortion condition} in the twisted cocycles. Roughly speaking, this condition asks for each of the terms in the left-hand side of \eqref{eq: fb} to be uniformly bounded. In particular, such condition is much more restrictive than our fiber-bunching assumption.

One can easily see that the $\alpha$-cohomology relation is an equivalence one over the space of $\alpha$-twisted cocycles. In particular, as a simple consequence of the previous result one can get a complete characterization of the cohomology classes in the twisted scenario in terms of the periodic data:
\begin{corollary}\label{cor: equivalence}
Let $f$, $A$, $B$ and $\alpha$ be as in Theorem \ref{teo: main} and, moreover, suppose that $A$ or $B$ satisfies \eqref{eq: fb}, \eqref{eq: alpha is Lips} and \eqref{eq: growth alpha} with $7\rho +2\theta <\nu \lambda$. Then, there exists a $\nu$-H\"older continuous map $Q:M\to GL(d,\mathbb{R})$ such that 
$$A^n_\alpha(p)=Q(p)B^n_\alpha(p)\alpha^n\left(Q(p)\right)^{-1}$$
for every $n\in \Z$ and $p\in \text{Fix}(f^n)$ if and only if there exists a $\nu$-H\"older continuous map $P:M\to GL(d,\mathbb{R})$ such that
\begin{equation*}
A^n_\alpha(x)=P\big(f^n(x)\big) B_\alpha^n (x)\alpha^n\left(P(x)\right)^{-1}, \quad\forall x\in M,\ \forall n\in\Z.
\end{equation*}
\end{corollary}
\begin{proof}
One implication is trivial. Let us deduce the other one. Assume that $B$ satisfies \eqref{eq: fb}, \eqref{eq: alpha is Lips} and \eqref{eq: growth alpha} with $7\rho +2\theta <\nu \lambda$. The case when $A$ satisfies it is similar. Let us consider 
\begin{displaymath}
\tilde{B}^n_\alpha(x)=Q(f^n(x))B^n_\alpha(x)\alpha^n(Q(x))^{-1}.
\end{displaymath}
We start observing that $\left(\tilde{B}^n_\alpha\right)_{n\in \Z}$ is an $\alpha$-twisted cocycle over $f$. Indeed,
\begin{displaymath}
\begin{split}
\tilde{B}_\alpha^{n+m}(x)&=Q(f^{n+m}(x))B_\alpha^{n+m}(x)\alpha^{m+n}(Q(x))^{-1}\\
&=Q(f^{n+m}(x))B_\alpha^{n}(f^m(x))\alpha^n\left( B_\alpha^m(x)\right)\alpha^{m+n}(Q(x))^{-1} \\
&=Q(f^{n+m}(x))B_\alpha^{n}(f^m(x))\alpha^n\left(Q(f^m(x))^{-1}Q(f^m(x))\right)\alpha^n\left( B_\alpha^m(x)\right)\alpha^{m+n}(Q(x))^{-1}\\
&=Q(f^{n+m}(x))B_\alpha^{n}(f^m(x))\alpha^n\left(Q(f^m(x)\right)^{-1} \alpha^n\left(Q(f^m(x)) B_\alpha^m(x)\alpha^{m}(Q(x))^{-1}\right)\\
&=\tilde{B}_\alpha^{n}(f^m(x))\alpha^n\left(\tilde{B}_\alpha^{m}(x)\right).
\end{split}
\end{displaymath}
Moreover, our hypothesis on $B$ ensures that $\tilde{B}$ is fiber-bunched in the sense of Section \ref{sec: FB}. Thus, since $A^n_\alpha(p)=\tilde{B}^n_\alpha(p)$ for every $p\in \text{Fix}(f^n)$ the result follows by applying our main result to these two cocycles. 
\end{proof}
Observe that the previous proof gives us no apparent ``meaningful" relation between the maps $P$ and $Q$ given in the statement of Corollary \ref{cor: equivalence}.

In order to proof our main result we follow the approaches of \cite{Bac15}, which in turn was inspired by \cite{Par99, Sch99}, and \cite{Sad15,Wal00}. The main idea consists in constructing \emph{invariant holonomies}, which is a family of linear maps with good properties (see Proposition \ref{prop: holonomies existence}), and then, using this family, to explicitly construct the transfer map on a dense set under the additional assumption that $f$ admits a fixed point. The next step consists in showing that, restricted to this dense set, the transfer map is $\nu$-H\"older continuous and then extending it to the whole space. Finally, we explain how to eliminate the hypothesis of existence of a fixed point for $f$. The main difference from this proof and the one in \cite{Bac15} is that the estimates here are much more involved due to the presence of twisting. The overall strategy is the same. In particular, the last step of the proof is the same, mutatis mutandis, as in the untwisted case and so we only indicate how to proceed.

Throughout the paper we are going to use the letter $C$ as a generic notation for a positive constant that may change from line to line. Whenever necessary, we will explicitly mention the parameters on which $C$ depends.

\section{Preliminaries}\label{sec: prelimin}

Let $(M,d)$ be a compact metric space, $f: M \to M $ a homeomorphism, $G$ a Lie group and $A:M\to G$ a $\nu$-H\"{o}lder continuous map. 

\subsection{Hyperbolic homeomorphisms} \label{sec: hyperbolic homeo}

Given any $x\in M$ and $\varepsilon >0$, define the \emph{local stable} and \emph{unstable sets} of $x$ with respect to $f$ by
\begin{align*}
  W^s_\varepsilon(x) &:= \left\{y\in M : d(f^n(x),f^n(y))\leq\varepsilon,\ \forall
    n \geq 0\right\}, \\
  W^u_\varepsilon(x) &:= \left\{y\in M : d(f^n(x),f^n(y))\leq\varepsilon,\ \forall
    n \leq 0\right\},
\end{align*}
respectively. Following \cite{AV10}, we introduce the following

\begin{definition}
  \label{def: hyperbolic homeo}
  A homeomorphism $f:M\to M$ is said to be \emph{hyperbolic with local product structure} (or just \emph{hyperbolic} for short) whenever there exist constants $C,\varepsilon,\lambda,\tau>0$ such that the following conditions are satisfied:
  \begin{itemize}
  \item[$\circ$] $d(f^n(y_1),f^n(y_2)) \leq Ce^{-\lambda n} d(y_1,y_2)$, $\forall x\in M$, $\forall y_1,y_2 \in W^s_\varepsilon (x)$, $\forall
    n\geq 0$;
  \item[$\circ$] $d(f^{-n}(y_1), f^{-n}(y_2)) \leq Ce^{-\lambda n} d(y_1,y_2)$, $\forall x\in M$, $\forall y_1,y_2 \in W^u_\varepsilon (x)$, $\forall n\geq 0$;
  \item[$\circ$] If $d(x,y)\leq\tau$, then $W^s_\varepsilon(x)$ and $W^u_\varepsilon(y)$ intersect in a unique point which is denoted by $[x,y]$ and depends continuously on $x$ and $y$.
  \end{itemize}
\end{definition}

For such homeomorphisms, one can define the \emph{stable} and \emph{unstable sets} by
\begin{displaymath}
  W^s(x):= \bigcup_{n\geq 0} f^{-n}\big(W^s_\varepsilon(f^n(x))\big) \quad\text{and}\quad W^u(x):= \bigcup_{n\geq 0} f^{n}\big(W^u_\varepsilon(f^{-n}(x))\big),
\end{displaymath}
respectively.

Notice that subshifts of finite type and basic sets of Axiom A diffeomorphisms are particular examples of hyperbolic homeomorphisms with local product structure (see for instance \cite[Chapter IV,\S~9]{Man87} for details).

\subsection{Twisted Cocycles}\label{sec: twisted}
Let $\Aut(G)$ denote the group of automorphisms of $G$ and $\alpha\in \Aut(G)$. A map $A_\alpha:\Z\times M\to G$ is said to be an \emph{$\alpha$-twisted cocycle} over $f$ if
\begin{equation*}
A_\alpha^{m+n}(x)=A_\alpha^n(f^m(x)) \alpha^n( A_\alpha^m(x))
\end{equation*}
for all $x\in M$ and $m,n \in \Z$. To any map $A:M\to G$ we may associate an $\alpha$-twisted cocycle over $f$ by
\begin{equation*}
A_\alpha^n(x)=
\left\{
	\begin{array}{ll}
		A(f^{n-1}(x))\alpha(A(f^{n-2}(x)))\ldots \alpha^{n-2}(A(f(x)))\alpha^{n-1}( A(x))  & \mbox{if } n>0 \\
		\Id & \mbox{if } n=0 \\
		\alpha^{n} ( A_\alpha^{-n} (f^{n}(x))^{-1}) & \mbox{if } n<0 \\
	\end{array}
\right.
\end{equation*}
for all $x\in M$. In this case we say that $A$ \emph{generates} the $\alpha$-twisted cocycle $A_\alpha$ over $f$. Reciprocally, every $\alpha$-twisted cocycle $A_\alpha$ is generated by $A=A^1_\alpha$. In what follows, for sake of simplicity, we write just $A_\alpha$ instead of $A^1_\alpha$.

\subsection{Cohomology of $\alpha$-twisted cocycles} Given a $\nu$-H\"{o}lder continuous map $B:M\to G$, we say that the $\alpha$-twisted cocycles $A_\alpha$ and $B_\alpha$ generated by $A$ and $B$ over $f$, respectively, are \emph{$\alpha$-cohomologous} if there exists a $\nu$-H\"{o}lder continuous map $P:M\to G$ such that
\begin{equation*}
A_\alpha (x)=P(f(x))B_\alpha(x)\alpha(P(x))^{-1}
\end{equation*}
for every $x\in M$. It is easy to verify that this equation is equivalent to 
\begin{equation*}
A^n_\alpha (x)=P(f^n(x))B^n_\alpha(x)\alpha^n(P(x))^{-1}
\end{equation*}
for every $x\in M$ and $n\in \Z$. As already observed in the introduction, whenever $\alpha=\Id$ we recover the usual notions of cocycles and cohomology \cite{Bac15,Sad15}.

\subsection{Linear $\alpha$-twisted cocycles} From now on we restrict ourselves to the case when $G=GL(d,\R)$. In particular, by $A:M\to GL(d,\mathbb{R})$ being $\nu$-H\"older continuous we mean that there exists a constant $C>0$ such that
\begin{equation}\label{eq: Holder}
\norm{A(x)-A(y)} \leq C d(x,y)^{\nu}
\end{equation}
for all $x,y\in M$ where $\norm{A}$ denotes the operator norm of a matrix $A$, that is, $\norm{A} =\sup \lbrace \norm{Av}/\norm{v};\; \norm{v}\neq 0 \rbrace$.

Observe that examples of automorphisms of $GL(d,\R)$ are $\alpha_L:GL(d,\R)\to GL(d,\R)$ and $\alpha_i:GL(d,\R)\to GL(d,\R)$ given by
$$\alpha_L(A)=LAL^{-1} \text{ and } \alpha_i(A)=\left(A^T\right)^{-1}$$
where $L\in GL(d,\R)$ is a fixed matrix and $A^T$ denotes the transpose of $A$. For more on $\Aut(GL(d,\R))$ we refer to \cite{Mcd78}.

\subsection{Fiber-bunched $\alpha$-twisted cocycles} \label{sec: FB} 
We say that the $\alpha$-twisted cocycle $A_\alpha$ generated by $A$ over $f$ is \emph{fiber-bunched} if there are constants $C>0$ and $\rho,\theta >0$ with $5\rho +2\theta <\nu \lambda$, where $\nu$ and $\lambda$ are as in \eqref{eq: Holder} and Definition \ref{def: hyperbolic homeo}, respectively, such that for every $n\in \Z$, 
\begin{itemize}
\item[i)]
\begin{equation}\label{eq: fb}
\|\alpha^{-n}(A^n_\alpha(x))\| \|\alpha^{-n}(A^n_\alpha(x)^{-1})\| < Ce^{\theta |n|}
\end{equation}
for every $x\in M$;
\item[ii)]
\begin{equation}\label{eq: alpha is Lips}
\| \alpha^n (T_1)- \alpha^n(T_2)\|\leq Ce^{\rho |n|} \|T_1-T_2\| 
\end{equation}
for every $T_1,T_2\in GL(d,\R)$;
\item[iii)]
\begin{equation}\label{eq: growth alpha}
\|\alpha^n(T)\| \leq Ce^{\rho |n|} \|T\|
\end{equation}
for every $T\in GL(d,\R)$.
\end{itemize}

Once again, it is easy to see that by taking $\alpha=\Id$ we recover the ``standard" notion of fiber-bunched cocycles used for instance in \cite{AV10,Bac15,BGV03,Sad15}. 

Observe that if $A$ and $\alpha$ are sufficiently close to the identity then the fiber-bunching condition is automatically satisfied. Other examples of $\alpha$-twisted cocycles with $\alpha\neq \Id$ satisfying the fiber-bunching condition are given, for instance, by taking $\alpha=\alpha_L$ as in the previous subsection with $L$ close enough to $\Id$ and assuming the cocycle $(A,f)$ is fiber-bunched in the standard sense of \cite{BGV03,Via08}. It is also worth noticing that this fiber-bunching notion is related to the partial hyperbolicity of the map $F_{A,\alpha}:M\times G\to M\times G$ given by $F_{A,\alpha}(x,g)=(f(x),A(x)\alpha(g))$. Indeed, condition \eqref{eq: fb} says that the rates of expansion and contraction given by $F_{A,\alpha}$ along the $G$-direction are ``dominated" by the rates of expansion and contraction along the $M$-direction.

\section{Invariant Holonomies}
In this section we introduce the notion of \emph{invariant holonomies} for twisted cocycles. This is done by generalizing the notion introduced by \cite{BGV03,Via08} in the untwisted case. As in the untwisted scenario, these objects are fundamental in our proof.

\begin{proposition}
  \label{prop: holonomies existence}
  Let $f: M\to M$ be a hyperbolic homeomorphism on a compact metric space $(M,d)$, $A: M\to GL(d,\mathbb{R})$ be a $\nu$-H\"older map and $\alpha\in \Aut(G)$. Suppose that the twisted cocycle $A_\alpha$ generated by $A$ and $\alpha$ over $f$ is fiber-bunched. Then there exists a constant $C=C(A,\alpha,f)>0$ such that, for any $x\in M$ and any $y, z\in W^s(x)$ the limit
  \begin{displaymath}
    H^{s,A,\alpha}_{yz}:= \lim_{n\to+\infty} \alpha^{-n}\left( A_\alpha^{n}(z)^{-1} A_\alpha^{n}(y) \right)
  \end{displaymath}
  exists and
  \begin{equation}\label{eq: holonom are holder}
  \|H^{s,A,\alpha}_{yz}- \Id \| \leq C d(y,z)^\nu,
  \end{equation}
  whenever $y,z\in W^s_\varepsilon(x)$, where the constant $\varepsilon>0$ associated to $f$ is given by Definition~\ref{def: hyperbolic homeo}.

  On the other hand, if $y, z\in W^u(x)$, we can analogously define
  \begin{displaymath}
    H^{u,A,\alpha}_{yz}:=\lim _{n\rightarrow +\infty} \alpha^n\left( A_\alpha^{-n}(z)^{-1}  A_\alpha^{-n}(y)\right), 
  \end{displaymath}
  and the very same H\"older estimates holds for these maps when $y,z\in W^u_\varepsilon(x)$.

  Finally, for every $x\in M$ and $*\in\{s,u\}$, it holds
  \begin{displaymath}
    H^{*,A,\alpha}_{yz}=H^{*,A,\alpha}_{xz} H^{*,A,\alpha}_{yx},
  \end{displaymath}
  and
  \begin{displaymath}
    H^{*,A,\alpha}_{f^m(y)f^m(z)} =A_\alpha^{m}(z) \alpha^m( H^{*,A,\alpha}_{yz})A_\alpha^m(y)^{-1}, 
  \end{displaymath}
  for every $y,z\in W^*(x)$ and $m\in \Z$.
\end{proposition}

\begin{definition}
  The maps $H^{s,A, \alpha}$ and $H^{u,A,\alpha}$ given by Proposition~\ref{prop: holonomies existence} are called \emph{stable} and \emph{unstable holonomies}, respectively.
\end{definition}

It is worth noticing that the main ideas beyond this concept, even though not under this name, were somehow present in \cite{Wal00} (see also \cite{Par99,Sch99} for the case $\alpha=Id$). On the other hand, the \emph{construction} of this holonomies in that setting is greatly simplified due to the existence of a bi-invariant metric. Similarly, the proof in the case $\alpha=\Id$ is also much simpler when compared to ours due to the lack of twisting (see for instance Proposition 2.5 of \cite{Via08}).

We will prove only the assertions about $H^{s,A,\alpha}_{y z}$ since the ones about $H^{u,A,\alpha}_{y z}$ are similar. We start with the following proposition:

\begin{proposition}\label{prop: auxiliary 1} 
Let $\delta>0$ be so that $5\rho+2\theta +\delta<\lambda\nu$. Then, there exists $C=C(A,\alpha, f,\delta)>0$ such that
\begin{displaymath}
\|\alpha^{-n} (A_\alpha^n(y))\|\cdot \|\alpha^{-n}( A_\alpha ^n(x)^{-1})\| \leq C e^{(4\rho+2\theta+\delta)n}
\end{displaymath}
for all $y\in W^s_{\varepsilon}(x)$, $x\in M$ and $n\geq 0$.
\end{proposition}

In order to prove this proposition we need a couple of auxiliary results.

\begin{lemma}\label{lemma: norms}
Fix $x\in M$. There exists a family of norms $(\|\cdot\|_k)_{k\in \N}$ such that
\begin{displaymath}
\frac{\max\left\lbrace\|\alpha^{-k}(A(f^{k-1}(x)))v\|_k; \; \|v\|_{k-1}=1\right\rbrace}{\min\left\lbrace\|\alpha^{-k}(A(f^{k-1}(x)))w\|_k; \; \|w\|_{k-1}=1\right\rbrace}\leq e^{2\theta +\delta}.
\end{displaymath}
Moreover, there exists $C>0$ depending only on $A$, $\alpha$, $f$ and $\delta$ so that
\begin{equation}\label{eq: relation norms}
\|\cdot\|\leq \|\cdot\|_k \leq Ce^{2\rho k} \|\cdot\| \text{ for every }k\in \N.
\end{equation}
\end{lemma}

\begin{proof}
Fix $u_0\in \R^d$ with $\|u_0\|=1$ and for any $k\in \Z$ set
\begin{equation*}
u_k=\frac{\alpha^{-k}(A(f^{k-1}(x)))u_{k-1}}{\|\alpha^{-k}(A(f^{k-1}(x)))u_{k-1}\|}.
\end{equation*}
Now, given $v\in \R^d$ define
\begin{equation}\label{eq: def norm k}
\|v\|_k^2=\sum_{m\in \Z} \frac{\|\alpha^{-m-k}(A^m_\alpha(f^k(x)))v\|^2}{\|\alpha^{-m-k}(A^m_\alpha(f^k(x)))u_k\|^2 \cdot e^{(2\theta+\delta)|m|}}.
\end{equation}
We start observing that from \eqref{eq: growth alpha}
\begin{displaymath}
\|\alpha^{-m-k}(A^m_\alpha(f^k(x)))\frac{v}{\|v\|}\|\|\alpha^{-m-k}(A^m_\alpha(f^k(x)))u_k\|^{-1}
\end{displaymath}
is smaller than or equal to
\begin{displaymath}
Ce^{2\rho k} \|\alpha^{-m}(A^m_\alpha(f^k(x)))\frac{v}{\|v\|}\|\|\alpha^{-m}(A^m_\alpha(f^k(x)))u_k\|^{-1}.
\end{displaymath}
Thus, using our hypothesis \eqref{eq: fb} and the fact that $\|T\|^{-1}\leq \|T^{-1}\|$ for any $T\in GL(d,\R)$ we get that the last quantity is smaller than or equal to $C^2e^{2\rho k}e^{\theta |m|}$. In particular,
\begin{displaymath}
\|\alpha^{-m-k}(A^m_\alpha(f^k(x)))v\|\|\alpha^{-m-k}(A^m_\alpha(f^k(x)))u_k\|^{-1}\leq C^2e^{2\rho k}e^{\theta |m|}\|v\|
\end{displaymath}
for every $m\in \Z$ and thus
\begin{displaymath}
\|v\|_k^2\leq \sum_{m\in \Z}\frac{\left( C^2e^{2\rho k}e^{\theta |m|}\|v\|\right)^2}{e^{(2\theta+\delta)|m|}}\leq \tilde{C}e^{4\rho k}\|v\|^2
\end{displaymath}
where $\tilde{C}=\sum_{m\in \Z}C^4e^{-\delta |m|}<\infty$. Consequently, the series \eqref{eq: def norm k} converges and $\|\cdot\|_k$ is well defined. Moreover
$$\|v\|_k\leq Ce^{2\rho k}\|v\|$$
for any $v\in \R^d$ and some $C>0$ independent of $k$ and $x$. Furthermore, recalling that $\alpha^{-k}(\Id)=\Id$ and $\|u_k\|=1$, looking at the term of \eqref{eq: def norm k} when $m=0$ it follows that $\|v\|\leq \|v\|_k$ for every $v\in \R^d$ which combined with the previous observations completes the proof of \eqref{eq: relation norms}. In order to prove the other claim, we observe that
\begin{displaymath}
\begin{split}
\|\alpha^{-k}(A(f^{k-1}(x))) v\|_k^2 &=\sum_{m\in \Z} \frac{\|\alpha^{-m-k}(A^m_\alpha(f^k(x)))\alpha^{-k}(A(f^{k-1}(x))) v\|^2}{\|\alpha^{-m-k}(A^m_\alpha(f^k(x)))u_k\|^2 \cdot e^{(2\theta+\delta)|m|}}\\
&=\sum_{m\in \Z} \frac{\|\alpha^{-m-k}(A^m_\alpha(f^k(x)))\alpha^{-k}(A(f^{k-1}(x))) v\|^2}{\|\alpha^{-m-k}(A^m_\alpha(f^k(x)))\left(\frac{\alpha^{-k}(A(f^{k-1}(x)))u_{k-1}}{\|\alpha^{-k}(A(f^{k-1}(x)))u_{k-1}\|}\right)\|^2 \cdot e^{(2\theta+\delta)|m|}}\\
&=\sum_{m\in \Z} \frac{\|\alpha^{-m-k}(A^{m+1}_\alpha(f^{k-1}(x))) v\|^2 \|\alpha^{-k}(A(f^{k-1}(x)))u_{k-1}\|^2}{\|\alpha^{-m-k}(A^{m+1}_\alpha(f^{k-1}(x)))u_{k-1}\|^2 \cdot e^{(2\theta+\delta)|m|}}\\
&=\|\alpha^{-k}(A(f^{k-1}(x)))u_{k-1}\|^2 \cdot S(v) \\
\end{split}
\end{displaymath}
where 
$$S(v):=\sum_{m\in \Z} \frac{\|\alpha^{-(m+1)-(k-1)}(A^{m+1}_\alpha(f^{k-1}(x))) v\|^2 }{\|\alpha^{-(m+1)-(k-1)}(A^{m+1}_\alpha(f^{k-1}(x)))u_{k-1}\|^2 \cdot e^{(2\theta+\delta)|m|}}.$$
Now, since $|m+1|\geq |m|-1$, we get that $S(v)\leq e^{2\theta+\delta}\|v\|_{k-1}^2$. Similarly, since $|m+1|\leq |m|+1$, we get that $S(v)\geq e^{-(2\theta+\delta)}\|v\|_{k-1}^2$. Combining these facts with the previous observations it follows that
\begin{displaymath}
\begin{split}
e^{-(\theta+\frac{\delta}{2})}\|\alpha^{-k}(A(f^{k-1}(x)))u_{k-1}\|\|v\|_{k-1}& \leq \|\alpha^{-k}(A(f^{k-1}(x))) v\|_k\\
&\leq e^{\theta+\frac{\delta}{2}}\|\alpha^{-k}(A(f^{k-1}(x)))u_{k-1}\|\|v\|_{k-1}
\end{split} 
\end{displaymath}
for any $v\in \R^d$. Thus, taking $v,w\in \R^d$ so that $\|v\|_{k-1}=\|w\|_{k-1}=1$ it follows that
\begin{displaymath}
\begin{split}
e^{-(2\theta+\delta)}\|\alpha^{-k}(A(f^{k-1}(x)))v\|_k &\leq \|\alpha^{-k}(A(f^{k-1}(x))) w\|_k \\
& \leq e^{2\theta+\delta}\|\alpha^{-k}(A(f^{k-1}(x)))v\|_k.
\end{split}
\end{displaymath}
Consequently,
\begin{displaymath}
\frac{\max\left\lbrace\|\alpha^{-k}(A(f^{k-1}(x)))v\|_k; \; \|v\|_{k-1}=1\right\rbrace}{\min\left\lbrace\|\alpha^{-k}(A(f^{k-1}(x)))w\|_k; \; \|w\|_{k-1}=1\right\rbrace}\leq e^{2\theta +\delta}
\end{displaymath}
as claimed.
\end{proof}

Thus, defining the \emph{$k$-norm} of an operator $T\in GL(d,\R)$ with respect to the family of norms $(\|\cdot\|_k)_{k\in \N}$ by
\begin{equation*}
\|T\|_k=\sup_{v\neq 0} \frac{\|Tv\|_k}{\|v\|_{k-1}}
\end{equation*}
it follows easily from the previous lemma that
\begin{corollary}\label{cor: growth adap norm} 
For any $k\in \N$, 
\begin{displaymath}
\|\alpha^{-k}(A(f^{k-1}(x)))^{-1}\|_k \|\alpha^{-k}(A(f^{k-1}(x)))\|_k\leq e^{2\theta+\delta}.
\end{displaymath}
\end{corollary}

\begin{proof}[Proof of Proposition \ref{prop: auxiliary 1}] Let $(\|\cdot\|_k)_{k\in \Z}$ be the family of norms given by Lemma \ref{lemma: norms}. Recalling \eqref{eq: alpha is Lips}, \eqref{eq: growth alpha} and \eqref{eq: relation norms}, we start observing that
\begin{displaymath}
\begin{split}
\frac{\|\alpha^{-k}(A(f^{k-1}(y)))\|_k}{\|\alpha^{-k}(A(f^{k-1}(x)))\|_k}&\leq 1+ \frac{\left|\|\alpha^{-k}(A(f^{k-1}(y)))\|_k-\|\alpha^{-k}(A(f^{k-1}(x)))\|_k \right| }{\|\alpha^{-k}(A(f^{k-1}(x)))\|_k} \\
& \leq 1+ \frac{\|\alpha^{-k}(A(f^{k-1}(y))) - \alpha^{-k}(A(f^{k-1}(x)))\|_k  }{\|\alpha^{-k}(A(f^{k-1}(x)))\|_k} \\
&\leq 1+ \frac{Ce^{2\rho k}\|\alpha^{-k}(A(f^{k-1}(y))) - \alpha^{-k}(A(f^{k-1}(x)))\|  }{\|\alpha^{-k}(A(f^{k-1}(x)))\|} \\
&\leq 1+ \frac{Ce^{4\rho k}\|A(f^{k-1}(y)) - A(f^{k-1}(x))\|  }{\|A(f^{k-1}(x))\|}. \\
\end{split}
\end{displaymath}
Thus, since $A$ is $\nu$-H\"older and $M$ is compact and recalling Definition \ref{def: hyperbolic homeo} it follows that
\begin{displaymath}
\begin{split}
\frac{\|\alpha^{-k}(A(f^{k-1}(y)))\|_k}{\|\alpha^{-k}(A(f^{k-1}(x)))\|_k}&\leq 1+ Ce^{(4\rho -\lambda \nu)k} d(x,y)^\nu. \\
\end{split}
\end{displaymath}
Now, Corollary \ref{cor: growth adap norm} gives us that for any $j\in \mathbb{N}$,
\begin{displaymath}
\|\alpha^{-j}(A(f^{j-1}(x)))^{-1}\|_j \leq \frac{e^{2\theta+\delta}}{\|\alpha^{-j}(A(f^{j-1}(x)))\|_j}.
\end{displaymath}
Combining these two observations with the fact that
\begin{displaymath}
\begin{split}
\|\alpha^{-k}(A^k_\alpha(x))^{-1}\|_k &= \|\alpha^{-1}(A(x))^{-1}\alpha^{-2}(A(f(x)))^{-1}\ldots\alpha^{-k}(A(f^{k-1}(x)))^{-1}\|_k \\
&\leq \|\alpha^{-1}(A(x))^{-1}\|_1 \|\alpha^{-2}(A(f(x)))^{-1}\|_2\ldots \|\alpha^{-k}(A(f^{k-1}(x)))^{-1}\|_k 
\end{split}
\end{displaymath}
and similarly
\begin{displaymath}
\|\alpha^{-k}(A^k_\alpha(y))\|_k \leq \|\alpha^{-k}(A(f^{k-1}(y)))\|_k \ldots \|\alpha^{-2}(A(f(x)))\|_2 \|\alpha^{-1}(A(x))\|_1
\end{displaymath}
it follows that
\begin{displaymath}
\begin{split}
\|\alpha^{-k}(A^k_\alpha(x))^{-1}\|_k\|\alpha^{-k}(A^k_\alpha(y))\|_k &\leq \frac{\|\alpha^{-1}(A(y))\|_1 }{\|\alpha^{-1}(A(x))\|_1}e^{2\theta+\delta} \ldots\frac{\|\alpha^{-k}(A(f^{k-1}(y)))\|_k }{\|\alpha^{-k}(A(f^{k-1}(x)))\|_k} e^{2\theta+\delta} \\
&\leq e^{(2\theta+\delta)k} \prod_{j=1}^k \left(1+ Ce^{(4\rho -\lambda \nu)j}d(x,y)^\nu  \right)\\
&\leq \tilde{C}e^{(2\theta+\delta)k}
\end{split}
\end{displaymath}
where $\tilde{C}=\prod_{j=1}^{\infty} \left(1+ CDe^{(4\rho -\lambda \nu)j}  \right)<\infty$ and $D=\sup_{x,y\in M}d(x,y)^\nu$ (recall that $4\rho -\lambda \nu <0$). Thus, since by \eqref{eq: relation norms} we have that $\|T\|\leq Ce^{2\rho k}\|T\|_k$ for every $T\in GL(d,\R)$ it follows that
$$\|\alpha^{-k}(A^k_\alpha(x))^{-1}\|\|\alpha^{-k}(A^k_\alpha(y))\|\leq Ce^{(4\rho+2\theta+\delta)k} $$
for some constant $C$ independent of $x$ and $y$ as claimed.
\end{proof}

We are now ready to prove the main proposition of this section.

\begin{proof}[Proof of Proposition \ref{prop: holonomies existence}] By taking forward iterates we can assume that $y,z\in W^u_{\frac{\varepsilon}{2}}(x)$. In particular, $z\in W^s_{\varepsilon}(y)$. We are going to show that the sequence $\left(\alpha^{-n}\left( A_\alpha^{n}(z)^{-1} A_\alpha^{n}(y)\right)\right)_n$ is a Cauchy sequence. In order to do it we start observing that for every $n\in \N$,
\begin{displaymath}
\begin{split}
\|\alpha^{-(n+1)}\left( A_\alpha^{n+1}(z)^{-1} A_\alpha^{n+1}(y)\right)-\alpha^{-n}\left( A_\alpha^{n}(z)^{-1} A_\alpha^{n}(y)\right)\|
\end{split}
\end{displaymath}
is equal to 
\begin{displaymath}
\begin{split}
\|\alpha^{-n}\left( A_\alpha^{n}(z)^{-1} \right) \alpha^{-(n+1)}\left( A(f^n(z))^{-1}A(f^n(y)) \right) \alpha^{-n}\left(  A_\alpha^{n}(y)\right)-\alpha^{-n}\left( A_\alpha^{n}(z)^{-1}\right) \alpha^{-n}\left(A_\alpha^{n}(y)\right)\|
\end{split}
\end{displaymath}
which is smaller than or equal to
\begin{displaymath}
\begin{split}
\|\alpha^{-n}\left( A_\alpha^{n}(z)^{-1} \right)\| \| \alpha^{-n}\left(A_\alpha^{n}(y)\right)\| \| \alpha^{-(n+1)}\left( A(f^n(z))^{-1}A(f^n(y)) \right) -\Id \|.
\end{split}
\end{displaymath}
From Proposition \ref{prop: auxiliary 1} it follows that the previous quantity is smaller than or equal to
\begin{displaymath}
 C e^{(4\rho+2\theta+\delta)n}  \| \alpha^{-(n+1)}\left( A(f^n(z))^{-1}A(f^n(y)) \right) -\Id \|.
\end{displaymath}
Thus, since 
\begin{displaymath}
\begin{split}
 \| \alpha^{-(n+1)}\left( A(f^n(z))^{-1}A(f^n(y)) \right) -\Id \|&= \| \alpha^{-(n+1)}\left( A(f^n(z))^{-1}A(f^n(y)) \right) -\alpha^{-(n+1)}(\Id) \| \\
 & \leq C e^{\rho (n+1)}\| A(f^n(z))^{-1}A(f^n(y)) -\Id \|\\
 & \leq Ce^{\rho (n+1)}e^{-\nu\lambda n}d(z,y)^\nu\\
 &= Ce^\rho e^{(\rho -\nu\lambda) n}d(z,y)^\nu
\end{split}
\end{displaymath}
we get that
\begin{displaymath}
\begin{split}
\|\alpha^{-(n+1)}\left( A_\alpha^{n+1}(z)^{-1} A_\alpha^{n+1}(y)\right)-\alpha^{-n}\left( A_\alpha^{n}(z)^{-1} A_\alpha^{n}(y)\right)\| &\leq C e^{(4\rho+2\theta+\delta)n}Ce^\rho e^{(\rho -\nu\lambda) n}d(z,y)^\nu\\
&=  C e^{(5\rho+2\theta+\delta -\nu\lambda) n}d(z,y)^\nu.\\
\end{split}
\end{displaymath}
Therefore, since $5\rho+2\theta+\delta -\nu\lambda<0$, we get that the sequence $\left(\alpha^{-n}\left( A_\alpha^{n}(z)^{-1} A_\alpha^{n}(y)\right)\right)_n$ is indeed a Cauchy sequence. Consequently 
\begin{displaymath}
    H^{s,A,\alpha}_{yz}= \lim_{n\to+\infty} \alpha^{-n}\left( A_\alpha^{n}(z)^{-1} A_\alpha^{n}(y) \right)
  \end{displaymath}
  exists and moreover
  \begin{displaymath}
\|H^{s,A,\alpha}_{yz}- \Id \| \leq C d(y,z)^\nu,
  \end{displaymath}
  whenever $y,z\in W^s_\varepsilon(x)$ as claimed.

To prove the last claim we start observing that, on the one hand,
\begin{displaymath}
\alpha^{-n}\left( A_\alpha^{n}(z)^{-1} A_\alpha^{n}(y) \right) \xrightarrow{n\to \infty}  H^{s,A,\alpha}_{yz}.
\end{displaymath}
On the other hand,
\begin{displaymath}
\alpha^{-n}\left( A_\alpha^{n}(z)^{-1} A_\alpha^{n}(y) \right) 
\end{displaymath}
is equal to 
\begin{displaymath}
\alpha^{-m}\left( A_\alpha^{m}(z)^{-1}\right) \alpha^{-m} \left( \alpha^{-(n-m)}\left(A_\alpha^{n-m}(f^m(z))^{-1} A_\alpha^{n-m}(f^m(y)) \right) \right) \alpha^{-m} \left( A_\alpha^{m}(y) \right)
\end{displaymath}
which converges to 
\begin{displaymath}
\alpha^{-m}\left( A_\alpha^{m}(z)^{-1}\right) \alpha^{-m}\left( H^{s,A,\alpha}_{f^m(y)f^m(z)} \right) \alpha^{-m} \left( A_\alpha^{m}(y) \right)
\end{displaymath}
as $n$ goes to infinity. Combining these observations we conclude that 
\begin{displaymath}
    H^{s,A,\alpha}_{f^m(y)f^m(z)} =A_\alpha^{m}(z) \alpha^m( H^{s,A,\alpha}_{yz})A_\alpha^m(y)^{-1}, 
  \end{displaymath}
as claimed.
\end{proof}

\begin{remark}\label{rem: aux inv hol}
From the proof of Proposition \ref{prop: auxiliary 1} we can easily see that in order to get
\begin{displaymath}
\|\alpha^{-k} (A_\alpha^k(y))\|\cdot \|\alpha^{-k}( A_\alpha ^k(x)^{-1})\| \leq C e^{(4\rho+2\theta+\delta)k}
\end{displaymath}
for every $0\leq k\leq n$ we don't actually need $y\in W^s_\varepsilon(x)$. In fact, we only need $x$ and $y$ to satisfy $d(f^k(x),f^k(y))\leq Ce^{-\gamma k}d(x,y)$ for every $0\leq k\leq n$ and some $\gamma \in (0,\lambda)$ satisfying $4\rho + \delta <\nu \gamma$. In this case, the constant $C$ will depend on $A$, $\alpha$, $f$, $\delta$ and $\gamma$. We are going to use this fact in the sequel.
\end{remark}

The notions of fiber-bunching and invariant holonomies in the case when $\alpha=\Id$ have been playing an important role in many subareas of Dynamical Systems and arise naturally in various different contexts (for instance, \cite{AV10,Bac15,BBB18,BGV03,Sad15,Via08}). Therefore, Proposition \ref{prop: holonomies existence} is also likely to have many applications and can be seen as interesting in itself.

In order to simplify notation, in what follows, whenever $\alpha$ is fixed and there is no ambiguity, we simply write $H^{*,A}$ instead of $H^{*,A,\alpha}$, for $*=s,u$, to denote the stable and unstable holonomy associated to $A_\alpha$.

\section{Constructing the transfer map}

In this section we are going to build ``explicitly" the transfer map. The method we use is similar to that used in \cite{Bac15,BK16} and \cite{Sad15} in the untwisted setting and by \cite{Wal00} in the twisted one: using the invariant holonomies we define the transfer map on a dense set, prove that restricted to it, it is H\"older continuous and then extend it to the closure getting the desired result.

Assume there exists $x\in M$ such that $f(x)=x$. For such a point, we write $W(x):=W^s(x)\cap W^u(x)$. We start defining $P\colon W(x)\to GL(d,\R)$ by
\begin{displaymath}
  P(y)=H^{s,A}_{xy} (H^{s,B}_{xy})^{-1}=H^{s,A,}_{xy} H^{s,B}_{yx},
\end{displaymath}
where $H^{s,A}$ and $H^{s,B}$ are the holonomy maps given by Proposition \ref{prop: holonomies existence} associated to the twisted cocycles $A_\alpha$ and $B_\alpha$, respectively.

Note that $P$ satisfies 
\begin{equation*}
  A_\alpha^n(y)=P(f^n(y)) B_\alpha^n(y) \alpha^n( P(y)^{-1})
\end{equation*}
for every $y\in W(x)$ and every $n\in\N$. Indeed, using that $f(x)=x$, Proposition \ref{prop: holonomies existence} and the hypothesis on periodic points \eqref{eq: periodic orbit cond main teo},
\begin{displaymath}
\begin{split}
P(f^n(y)) & = H^{s,A}_{xf^n(y)} H^{s,B}_{f^n(y)x} = H^{s,A}_{f^n(x)f^n(y)} H^{s,B}_{f^n(y)f^n(x)}\\
& =A^n_\alpha(y) \alpha^n(H^{s,A}_{xy})A^n_\alpha(x)^{-1} B_\alpha^n(x) \alpha^n(H^{s,B}_{yx})B_\alpha ^n(y)^{-1} \\
& =A^n_\alpha(y) \alpha^n(H^{s,A}_{xy} H^{s,B}_{yx})B_\alpha ^n(y)^{-1} \\
&=A^n_\alpha(y) \alpha^n(P(y))B_\alpha ^n(y)^{-1} \\
\end{split}
\end{displaymath}
and thus
$$ A_\alpha^n(y)=P(f^n(y)) B_\alpha^n(y) \alpha^n( P(y)^{-1})$$
as claimed.

We will now show that $P$ is $\nu$-H\"older continuous. This will allow us to extend $P$ to $\overline{W(x)}=M$ and thus to get the desired transfer map. The main ingredient in the proof is the next lemma which says that $P$ can be interchangeably defined using stable or unstable holonomies. Its proof is similar to the one of \cite[Lemma 3]{Bac15} and we only present the full details of it because of its main role in our proof and also because the presence of twist makes some estimates a little more involved than in the untwisted case.

\begin{lemma}\label{lemma: P su holon}
For every $y\in W(x)$,
  \begin{displaymath}
    P(y)=H^{s,A}_{xy} H^{s,B}_{yx}= H^{u,A}_{xy} H^{u,B}_{yx}.
  \end{displaymath}
\end{lemma}

The following classical result (see for instance \cite[Corollary 6.4.17]{KH95}) will be used in the proof:

\begin{lemma}[Anosov Closing Lemma] Given $\gamma\in (0,\lambda)$ there exist $C>0$ and $\varepsilon _0>0$ such that if $z\in M$ satisfy $d(f^n(z),z)<\varepsilon _0$ then there exists a periodic point $p\in M$ such that $f^n(p)=p$ and
\begin{equation} \label{eq: Anosov CL}
d(f^j(z),f^j(p))\leq C e^{-\gamma \min\lbrace j, n-j\rbrace}d(f^n(z),z)
\end{equation} 
for every $j=0,1,\ldots ,n$.
\end{lemma}

\begin{proof}[Proof of Lemma \ref{lemma: P su holon}]

Let $\delta>0 $ be so that $5\rho+2\theta+\delta<\lambda \nu$ and $\gamma \in (0,\lambda)$ such that $5\rho+2\theta+\delta<\gamma \nu$. Let $C>0$ and $\varepsilon _0>0$ be given by the Anosov Closing Lemma associated to $\gamma$.

Fix an arbitrary point $y\in W(x)$. We begin by noticing that, as $y\in W(x)$, there exist $C>0$ and $n_0 \in \mathbb{N}$ such that for all $n\geq n_0$ we have
\begin{center}
$d(f^{-n}(y),f^n(y))\leq Ce^{-\lambda(n-n_0)}$.
\end{center}
In fact, this follows from the fact that, as $y\in W(x)=W^s(x)\cap W^u(x)$, there exists $n_0\in \mathbb{N}$ such that $f^{n_0}(y)\in W^s_{\varepsilon }(x)$ and $f^{-n_0}(y)\in W^u_{\varepsilon }(x)$ and the exponential convergence towards $x$ in $W^s_{\varepsilon}(x)$ and $W^u_{\varepsilon}(x)$. 

Let $n_1\geq n_0$ be such that, for all $n\geq n_1$, $d(f^n(y), f^{-n}(y))<\varepsilon _0$. Thus, by the Anosov Closing Lemma, for every $n\geq n_1$ there exists a periodic point $p_n\in M$ with $f^{2n}(p_n)=p_n$ such that
\begin{displaymath}
	d(f^j(f^{-n}(p_n)), f^j(f^{-n}(y))\leq C e^{-\gamma\min \lbrace j, 2n-j\rbrace }d(f^{-n}(y),f^n(y))
\end{displaymath}
for every $j=0,1,\ldots ,2n$. Using the periodic orbit condition \eqref{eq: periodic orbit cond main teo} and noticing that $f^{2n}(f^{-n}(p_n))=f^{-n}(p_n)$, we get
\begin{displaymath}
A_\alpha^{2n}(f^{-n}(p_n))=B_\alpha^{2n}(f^{-n}(p_n)),
\end{displaymath}
which can be rewritten as 
\begin{displaymath}
A_\alpha^n(p_n)\alpha^n\left(A_\alpha^n(f^{-n}(p_n))\right)=B_\alpha^n(p_n)\alpha^n\left(B_\alpha^n(f^{-n}(p_n))\right),
\end{displaymath}
or, equivalently, as
\begin{displaymath}
\alpha^n\left(A_\alpha^n(f^{-n}(p_n))B_\alpha^n(f^{-n}(p_n))^{-1}\right)=A_\alpha^n(p_n)^{-1}B_\alpha^n(p_n).
\end{displaymath}
Thus, observing that 
\begin{displaymath}
	A^n_\alpha(f^{-n}(p_n))=\alpha^n\left(A^{-n}_\alpha(p_n)^{-1}\right)
\end{displaymath}
and similarly
\begin{displaymath}
	B^n_\alpha(f^{-n}(p_n))^{-1}=\alpha^n\left(B^{-n}_\alpha(p_n)\right)
\end{displaymath}
we get 
\begin{equation}\label{eq: aux 1 P su}
\alpha ^n\left(A_\alpha^{-n}(p_n)^{-1}B_\alpha^{-n}(p_n)\right)= \alpha^{-n} \left(A_\alpha^n(p_n)^{-1}B_\alpha^n(p_n)\right).
\end{equation}
Now we claim that 
\begin{equation}\label{eq: claim P su}
\| \alpha^{-n} \left(A_\alpha^n(y)^{-1}B_\alpha^n(y)\right)-\alpha^{-n} \left(A_\alpha^n(p_n)^{-1}B_\alpha^n(p_n)\right)\|\xrightarrow{n\to +\infty}0
\end{equation}
and 
\begin{equation}\label{eq: claim P su2}
\| \alpha ^n\left(A_\alpha^{-n}(y)^{-1}B_\alpha^{-n}(y)\right) - \alpha ^n\left(A_\alpha^{-n}(p_n)^{-1}B_\alpha^{-n}(p_n)\right)\|\xrightarrow{n\to +\infty}0.
\end{equation}
Consequently, it follows from \eqref{eq: aux 1 P su} and our claim that 
\begin{displaymath}
\| \alpha^{-n} \left(A_\alpha^n(y)^{-1}B_\alpha^n(y)\right)- \alpha ^n\left(A_\alpha^{-n}(y)^{-1}B_\alpha^{-n}(y)\right)\|\xrightarrow{n\to +\infty}0.
\end{displaymath}
Thus, observing that
\begin{displaymath}
\begin{split}
\alpha^{-n} \left(A_\alpha^n(y)^{-1}B_\alpha^n(y)\right)&=\alpha^{-n} \left(A_\alpha^n(y)^{-1}A_\alpha^n(x)B_\alpha^n(x)^{-1}B_\alpha^n(y)\right)\xrightarrow{n\to +\infty} H^{s,A}_{xy}H^{s,B}_{yx}
\end{split}
\end{displaymath}
and similarly
\begin{displaymath}
	\alpha ^n\left(A_\alpha^{-n}(y)^{-1}B_\alpha^{-n}(y)\right) \xrightarrow{n\to +\infty} H^{u,A}_{xy}H^{u,B}_{yx}
\end{displaymath}
we conclude that
\begin{displaymath}
P(y)=H^{s,A}_{xy}H^{s,B}_{yx}=H^{u,A}_{xy}H^{u,B}_{yx}
\end{displaymath}	
as we wanted.

So, in order to complete the proof, it remains to prove our claim. We shall only prove \eqref{eq: claim P su} since \eqref{eq: claim P su2} is completely analogous.

We start observing that
\begin{displaymath}
\|\alpha^{-n}\left(A_\alpha^{n}(y)A_\alpha^{n}(p_n)^{-1}\right)-\Id\|
\end{displaymath}
is smaller than or equal to
\begin{multline*}
 \sum ^{n-1}_{j=0} \|\alpha^{-(n-j)} \left( A_\alpha^{n-j}(f^j(y))A_\alpha^{n-j}(f^j(p_n))^{-1}\right) \\
 -\alpha^{-(n-j)} \left(A_\alpha^{n-j-1}(f^{j+1}(y))A_\alpha^{n-j-1}(f^{j+1}(p_n))^{-1}\right)\|
\end{multline*}
which by the cocycle property \eqref{eq: twisted cocycle eq intro} is equal to
\begin{multline*}
\sum ^{n-1}_{j=0}\| \alpha^{-(n-j)} \left(A_\alpha^{n-j-1}(f^{j+1}(y))\right) \alpha^{-1}\left(A(f^j(y))A(f^j(p_n))^{-1}\right) \alpha^{-(n-j)}\left(A_\alpha^{n-j-1}(f^{j+1}(p_n))^{-1}\right) \\
- \alpha^{-(n-j)} \left(A_\alpha^{n-j-1}(f^{j+1}(y))A_\alpha^{n-j-1}(f^{j+1}(p_n))^{-1}\right)\| .
\end{multline*}
By the property of the norm this last quantity is smaller than or equal to
\begin{multline*}
\sum ^{n-1}_{j=0}\| \alpha^{-(n-j)} \left(A_\alpha^{n-j-1}(f^{j+1}(y))\right)\| \|\alpha^{-(n-j)}\left(A_\alpha^{n-j-1}(f^{j+1}(p_n))^{-1}\right)\|\\
 \cdot  \allowbreak \|\alpha^{-1}\left(A(f^j(y))A(f^j(p_n))^{-1}\right)-\Id\|
\end{multline*}
which in its turn is smaller than or equal to
\begin{multline*}
\sum ^{n-1}_{j=0}C^2e^{2\rho} \|\alpha^{-(n-j-1)} \left(A_\alpha^{n-j-1}(f^{j+1}(y))\right)\| \|\alpha^{-(n-j-1)}\left(A_\alpha^{n-j-1}(f^{j+1}(p_n))^{-1}\right)\|  \\
\cdot\|A(f^j(y))A(f^j(p_n))^{-1}-\Id\|.
\end{multline*}
Now, using Remark \ref{rem: aux inv hol}, the fact that $A$ is $\nu$-H\"older continuous and property \eqref{eq: Anosov CL} given by the Anosov Closing Lemma, it follows that the previous quantity is smaller than or equal to 
\begin{displaymath}
\sum ^{n-1}_{j=0}C^2e^{2\rho} Ce^{(4\rho+2\theta+\delta)(n-j-1)} Ce^{-\gamma \nu (n-j-1)} d(f^{-n}(y),f^{n}(y))^\nu.
\end{displaymath}
 Recalling that $d(f^{-n}(y),f^{n}(y))\leq e^{-\lambda(n-n_0)}$ for every $n\geq n_0$ and $5\rho+2\theta+\delta<\gamma \nu $ it follows that
\begin{equation}\label{eq: aux 2 P su}
\|\alpha^{-n}\left(A_\alpha^{n}(y)A_\alpha^{n}(p_n)^{-1}\right)-\Id\| \leq Ce^{-\lambda \nu (n-n_0)}
\end{equation}
for every $n\geq n_0$ for some constant $C>0$ independent of $n$ and $p_n$. Similarly, we can prove that 
\begin{equation}\label{eq: aux 3 P su}
\|\alpha^{-n}\left(B_\alpha^{n}(p_n)A_\alpha^{n}(y)^{-1}\right)-\Id\| \leq Ce^{-\lambda \nu (n-n_0)}
\end{equation}
for every $n\geq n_0$.

Now, using that there exists $N>0$ so that $\|\alpha^{-n} \left(A_\alpha^n(y)^{-1}A_\alpha^n(x)\right)\|<N$ and $\|\alpha^{-n} \left(B_\alpha^n(x)^{-1}B_\alpha^n(y)\right)\|<N$ for every sufficiently large $n$ since these two quantities converge to $H^{s,A}_{xy}$ and $H^{s,B}_{yx}$, respectively; $\|\alpha^{-n}\left(A_\alpha^{n}(y)A_\alpha^{n}(p_n)^{-1}\right)\|<N$ for every sufficiently large $n$ by \eqref{eq: aux 2 P su}; $A^n_\alpha(x)=B^n_\alpha(x)$ since $A_\alpha$ and $B_\alpha$ satisfy periodic orbit condition \eqref{eq: periodic orbit cond main teo} and $f(x)=x$; and using \eqref{eq: fb}, \eqref{eq: aux 2 P su} and \eqref{eq: aux 3 P su}, we get
\begin{displaymath}
\begin{split}
&\| \alpha^{-n} \left(A_\alpha^n(y)^{-1}B_\alpha^n(y)\right)-\alpha^{-n} \left(A_\alpha^n(p_n)^{-1}B_\alpha^n(p_n)\right)\|\\
&=\| \alpha^{-n} \left(A_\alpha^n(y)^{-1}A_\alpha^n(x)\right)\alpha^{-n} \left(B_\alpha^n(x)^{-1}B_\alpha^n(y)\right)-\alpha^{-n} \left(A_\alpha^n(p_n)^{-1}B_\alpha^n(p_n)\right)\| \\
&\leq N^2 \|\Id -\alpha^{-n} \left(A_\alpha^n(x)^{-1}A_\alpha^n(y)A_\alpha^n(p_n)^{-1}B_\alpha^n(p_n)B_\alpha^n(y)^{-1}B_\alpha^n(x)\right)\| \\
&\leq N^2 \|\alpha^{-n} \left(A_\alpha^n(x)^{-1}\right)\|\|\alpha^{-n} \left(A_\alpha^n(x)\right)\| \|\Id -\alpha^{-n} \left(A_\alpha^n(y)A_\alpha^n(p_n)^{-1}B_\alpha^n(p_n)B_\alpha^n(y)^{-1}\right)\| \\
&\leq  N^2 Ce^{\theta n} \Big(\|\alpha^{-n} \left(A_\alpha^n(y)A_\alpha^n(p_n)^{-1}\right)\|\| \alpha^{-n} \left(B_\alpha^n(p_n)B_\alpha^n(y)^{-1}\right)-\Id\| \\
& +\|\alpha^{-n} \left(A_\alpha^n(y)A_\alpha^n(p_n)^{-1}\right)-\Id\| \Big) \\
&\leq N^2 Ce^{\theta n} \left(NCe^{-\lambda \nu (n-n_0)}+Ce^{-\lambda \nu (n-n_0)}\right)\\
&\leq \tilde{C}e^{(\theta-\lambda)n} 
\end{split}
\end{displaymath}
for some $\tilde{C}>0$ independent of $n$ and $p_n$ and $n\gg0$. In particular,
\begin{displaymath}
\| \alpha^{-n} \left(A_\alpha^n(y)^{-1}B_\alpha^n(y)\right)-\alpha^{-n} \left(A_\alpha^n(p_n)^{-1}B_\alpha^n(p_n)\right)\| \xrightarrow{n\to +\infty}0
\end{displaymath}
proving \eqref{eq: claim P su} and thus completing the proof Lemma \ref{lemma: P su holon}.
\end{proof}

\begin{lemma}\label{lemma: P is Lipschitz}
$P$ is $\nu$-H\"older continuous on $W(x)$.
\end{lemma}
\begin{proof}
The proof of this fact is analogous to the proof of \cite[Lemma 4]{Bac15} and so we just summarize the idea beyond it. Full details can be checked in the original work. 

From Lemma \ref{lemma: P su holon} we know that $P$ can be defined using both stable and unstable holonomies. By property \eqref{eq: holonom are holder} we get that restricted to local stable or unstable manifolds, $P$ is H\"older continuous with an uniform H\"older constant. Now, since $f$ has local product structure, points that are $\tau$-close (where $\tau$ is as in Definition \ref{def: hyperbolic homeo}) can be connected via local stable and unstable manifolds. Putting all these facts together we conclude that $P$ is H\"older continuous on balls of radius $\tau$. Finally, using the compactness of $M$ we conclude that $P$ is H\"older continuous in $W(x)$.
\end{proof}

Therefore, we can extend $P: W(x)\to GL(d,\mathbb{R}))$ to the closure of $W(x)$ that is the whole space $M$. By continuity, such extension clearly satisfies the cohomological equation \eqref{eq: cohomo eq main teo} completing the proof of Theorem \ref{teo: main} in the case when $f$ exhibits a fixed point. 

Now, following the argument given in Section 5 of \cite{Bac15}, mutatis mutandis, we eliminate the additional assumption about the existence of a fixed point for $f$ and conclude the proof of Theorem \ref{teo: main}.


\medskip{\bf Acknowledgements.} Thanks to the referee for the constructive comments on the first version of the paper. The author was partially supported by a CNPq-Brazil PQ fellowship under Grant No. 306484/2018-8.


\end{document}